\documentclass[11pt]{article}
\usepackage[a4paper,textwidth=165mm,textheight=250mm]{geometry}
\usepackage[applemac]{inputenc}
\usepackage[T1]{fontenc}
\usepackage{lmodern}
\usepackage{amssymb}
\usepackage{amsmath}
\usepackage{latexsym}
\usepackage{theorem}
\usepackage{eucal}
\usepackage[sans]{dsfont}
\usepackage{enumitem}
\usepackage[matrix,arrow,curve,cmtip]{xy}
\usepackage{titlesec}
\usepackage{hyperref}

\allowdisplaybreaks

\titleformat{\section}[hang]{\bf\Large}{\thesection.}{1ex}{}
\titleformat{\subsection}[hang]{\bfseries\normalsize}{\thesubsection}{1ex}{}

\def\distsign{\begin{picture}(0,0)\put(0,0){\circle{4}}\end{picture}}
\def\dist{\mbox{$\xymatrix@1@C=5mm{\ar@{->}[r]|{\distsign}&}$}}
\newcommand\adj[2]{\xymatrix@C=8ex{\ar@{}[r]|{\perp}\ar@/^2ex/[r]^{#1} & \ar@/^2ex/[l]^{#2}}}
\newcommand\jda[2]{\xymatrix@C=8ex{\ar@{}[r]|{\top}\ar@/^2ex/[r]^{#1} & \ar@/^2ex/[l]^{#2}}}
\newcommand\arr[1]{\xymatrix{\ar@{->}[r]^{#1}&}}

\newtheorem{theorem}{Theorem}[subsection]
\newtheorem{definition}[theorem]{Definition}
\newtheorem{proposition}[theorem]{Proposition}
\newtheorem{corollary}[theorem]{Corollary}

{\theorembodyfont{\upshape}\newtheorem{example}[theorem]{Example}}
{\theorembodyfont{\upshape}}
{\theorembodyfont{\upshape}}
\newenvironment{proof}[1][Proof.]{\begin{trivlist}\item[\hskip\labelsep{\it #1}]}{\hfill$\Box$\end{trivlist}}

\newcommand\Sup{\mathsf{Sup}}
\newcommand\N{\mathbb{N}}
\renewcommand\:{\colon}
\newcommand\Cat{{\sf Cat}}
\newcommand\Dist{{\sf Dist}}
\newcommand\Q{Q}
\newcommand\bbA{\mathbb{A}}
\newcommand\bbB{\mathbb{B}}
\newcommand\bbC{\mathbb{C}}
\newcommand\bbD{\mathbb{D}}
\renewcommand\1{\mathds{1}}
\newcommand\lb{\hspace{0.3ex}[\hspace{-1ex}[\hspace{0.7ex}}
\newcommand\rb{\hspace{0.7ex}]\hspace{-1ex}]\hspace{0.3ex}}
\newcommand\LAdj{{\sf LAdj}}

\title{A logical analysis of fixpoint theorems}
\author{Arij Benkhadra\footnote{Laboratoire de Mathématiques Pures et Appliquées, Université du Littoral, Calais, France; Centre de Recherche de Mathématiques, Université Mohammed V, Rabat, Maroc; {\tt arij.benkahdra@univ-littoral.fr}}\quad and\quad Isar Stubbe\footnote{Laboratoire de Mathématiques Pures et Appliquées, Université du Littoral, Calais, France; {\tt isar.stubbe@univ-littoral.fr}}}
\date{October 27, 2022}

\begin{document}

\maketitle

\begin{abstract}
We prove a fixpoint theorem for contractions on Cauchy-complete quantale-enriched categories. It holds for any quantale whose underlying lattice is continuous, and applies to contractions whose control function is sequentially lower-semicontinuous. Sufficient conditions for the uniqueness of the fixpoint are established. Examples include known and new fixpoint theorems for metric spaces, fuzzy metric spaces, and probabilistic metric spaces.
\end{abstract}

\section*{Introduction}

A beautiful and important result in metric space theory, is Banach's fixpoint theorem \cite{banach} from 1922: ``Every contraction on a non-empty complete metric space admits a unique fixpoint.'' The gist of the proof is wonderfully simple: take any element $x$ of the space $(X,d)$ and, iterating the contraction $f\:X\to X$, prove that the sequence $(f^nx)_{n\in\N}$ is Cauchy. In the complete space $(X,d)$ this sequence converges, and one then shows that it does so to a (necessarily unique) fixpoint of $f$. Many generalizations and applications of Banach's theorem have been, and are still, studied. 

In 1972, Lawvere \cite{Lawvere} famously showed that metric spaces are a particular instance of enriched categories. More impressively still, Lawvere also showed how convergence of Cauchy sequences can adequately be understood via representability of left adjoint distributors, thus lifting the very concept of Cauchy completeness to the level of enriched categories. In his words, ``specializing the constructions and theorems of general category theory we can deduce a large part of general metric space theory.'' 

It is thus natural to investigate whether fixpoint theorems still make sense in the vast context of enriched categories. This is precisely the subject of this paper. 

More precisely, we shall take quantale-enriched categories as generalization of metric spaces. That is to say, we fix a quantale $Q$, and work with categories, functors and distributors enriched in $Q$. Our contribution shows that fixpoint theorems for $Q$-categories depend on the interplay between three essential parameters. Indeed, a given contraction must be ``strong enough'' (we shall measure its strength by means of a control function); the space on which it acts must be ``complete enough'' for the Picard iteration to converge to a fixpoint (we shall take this to be Cauchy-completeness in the sense of Lawvere); but we also need sufficiently strong algebraic properties of the underlying quantale $Q$ to allow for the formulation of precisely that convergence.

In concreto, we shall prove a fixpoint theorem for Cauchy-complete $Q$-categories\footnote{To stay faithful to Banach's theorem in the metric case, we have chosen to study fixpoints for contractions on \emph{Cauchy-complete} $Q$-categories. Let us mention, though, that other authors have studied other kinds of completeness, e.g.\ Wagner \cite{wagner97} chooses liminf-complete $Q$-categories, whereas Ackerman \cite{Ackerman16} works with spherically complete $Q$-categories (and both use a commutative quantale $Q$).} that holds for any quantale $Q$ whose underlying complete lattice is continuous and for a specific notion of contraction. Besides, we make plain when and why such a fixpoint is unique (up to isomorphism). As examples we find the classical Banach fixpoint theorem for metric spaces, and Boyd and Wong's \cite{boywan69} generalization thereof (taking the underlying quantale to be the positive real numbers); but we also formulate new results for fuzzy ordered sets (when working over a left-continuous $t$-norm) and for probabilistic metric spaces (now the quantale is the tensor product of the positive reals with a left-continuous $t$-norm).

In Section \ref{cats} we shall provide all the necessary notions from quantale-enriched category theory to make this paper reasonably self-contained; we follow \cite{stu05} for the general theory, and \cite{hofrei13} specifically for the comparison between categorical and sequential Cauchy-completeness. In Section \ref{fix} we first introduce the contractions that we are interested in, then we show how these contractions determine Cauchy distributors under the appropriate algebraic condition on the quantale $Q$, and finally we formulate the resulting fixpoint theorem for Cauchy-complete $Q$-categories. The examples in Section \ref{examples} show how our fixpoint theorem generalizes known results from the literature, and provides for new results too. We end with a short conclusion in Section \ref{concl}.

\section{Quantale-enriched categories}\label{cats}

\subsection{\texorpdfstring{$Q$}{Q}-enriched categories, functors and distributors}

In this section we recall some key notions from \cite{stu05} on quantale-enriched categories\footnote{That reference actually treats the more general \emph{quantaloid}-enriched category theory, but the reader will easily convert those results to the simpler quantale-enriched case. See also \cite{stu13} for a gentle introduction to the subject.}; we encourage the reader to go back-and-forth to Subsection \ref{examples-Q-cat} for the relevant examples. 

Throughout, we fix a quantale $\Q=(\Q,\bigvee,\circ,1)$: it is a complete sup-lattice $(\Q,\bigvee)$ endowed with a monoid\footnote{We do \emph{not} assume that $1$, the unit of the monoid, is the top element of the lattice.} structure $(\Q,\circ,1)$ such that the product distributes over arbitrary suprema:
$$s\circ(\bigvee_i t_i)=\bigvee_i(s\circ t_i)\quad\mbox{ and }\quad(\bigvee_is_i)\circ t=\bigvee_i(s_i\circ t).$$
In other words, but more abstractly, a quantale is a monoid in the symmetric monoidal closed category $\Sup$ of complete lattices and supremum-preserving morphisms.

A \emph{$\Q$-enriched category $\mathbb{C}$} (or \emph{$\Q$-category $\bbC$} for short) consists of a set $\mathbb{C}_0$ (of ``objects'') together with a $Q$-valued (``hom'') precidate 
$$\bbC\:\bbC_0\times\bbC_0\to\Q\:(x,y)\mapsto\bbC(x,y)$$
satisfying, for all $x,y,z\in\bbC_0$, the following (``composition'' and ``identity'') conditions:
$$\mathbb{C}(x,y)\circ\mathbb{C}(y,z)\leq\mathbb{C}(x,z)\quad\mbox{ and }\quad 1\leq\mathbb{C}(x,x).$$

A \emph{$\Q$-functor} $F\:\bbC\to\bbD$ between two $\Q$-categories is a function $F\:\bbC_0\to\bbD_0\:x\mapsto Fx$
satisfying, for all $x,x'\in\bbC_0$, the (``functoriality'') condition
$$\bbC(x',x)\leq\bbD(Fx',Fx).$$
Two such $\Q$-functors $F\:\bbA\to\bbB$ and $G\:\bbB\to\bbC$ can be composed in the obvious way to produce a new functor $G\circ F\:\bbA\to\bbC$, and the identity function on $\bbA_0$ provides for the identity functor $1_{\bbA}\:\bbA\to\bbA$. Thus $\Q$-categories and $\Q$-functors are the objects and morphisms of a (large) category $\Cat(\Q)$. 

A \emph{$\Q$-distributor} (also called \emph{bimodule} or \emph{profunctor}) $\Phi\:\bbC\dist\bbD$ between two $\Q$-categories is a $\Q$-valued predicate
$$\Phi\:\bbD_0\times\bbC_0\to\Q\:(y,x)\mapsto\Phi(y,x)$$
satisfying, for all $x,x'\in\bbC_0$ and $y,y'\in\bbC_0$, the (``action'') condition
$$\bbD(y',y)\circ\Phi(y,x)\circ\bbC(x,x')\leq\Phi(y',x').$$
Two such distributors, say $\Phi\:\bbA\dist\bbB$ and $\Psi\:\bbB\dist\bbC$, compose as
$$(\Psi\circ\Phi)\:\bbC_0\times\bbA_0\mapsto\Q\:(z,x)\mapsto\bigvee_{y\in\bbB_0}\Psi(z,y)\circ\Phi(y,x).$$
The identity distributor on $\bbC$ is the ``hom'' predicate $\bbC\:\bbC_0\times\bbC_0\to\Q$ itself, and so $\Q$-categories and $\Q$-distributors form a (large) category $\Dist(\Q)$. However, there is more: the elementwise ordering of distributors makes $\Dist(Q)$ a 2-category\footnote{Much better still: $\Dist(\Q)$ is a \emph{quantaloid}, i.e.\ a category enriched in $\Sup$. Since we do not need this very rich structure in this paper, we shall not dwell on it here.}. 

Applying general 2-categorical algebra, we may now say that two $\Q$-distributors $\Phi\:\bbA\dist\bbB$ and $\Psi\:\bbB\dist\bbA$ are \emph{(left/right) adjoint}, denoted $\Phi\dashv\Psi$, if
$$\bbA\leq\Psi\circ\Phi\quad\mbox{ and }\quad\Phi\circ\Psi\leq\bbB.$$
Every functor $F\:\bbA\to\bbB$ \emph{represents} an adjoint pair of distributors $F_*\dashv F^*$ defined by
$$F_*(b,a)=\bbB(b,Fa)\quad\mbox{ and }\quad F^*(a,b)=\bbB(Fa,b).$$ 
With this, the inclusion functor
$$\Cat(\Q)\to\Dist(\Q)\:\Big(F\:\bbA\to\bbB\Big)\mapsto\Big(F_*\:\bbA\dist\bbB\Big)$$
naturally makes $\Cat(\Q)$ a \emph{locally ordered category} by defining, for $F,G\in\Cat(Q)$,
$$F\leq G\stackrel{\rm def}{\iff}F_*\leq G_*.$$
Whenever $F\leq G$ and $G\leq F$, we write $F\cong G$ and say that these functors are \emph{isomorphic}.

For a fixed $Q$-category $\bbC$, we may consider, for any other $Q$-category $\bbA$, the map which assigns to any functor $F\:\bbA\to\bbC$ the left adjoint distributor $F_*\:\bbA\dist\bbC$:
$$\Cat(Q)(\bbA,\bbC)\to\LAdj\Dist(Q)(\bbA,\bbC)\:F\mapsto F_*.$$
This map is (by definition of the local order in $\Cat(Q)$) order-preserving and order-reflecting. If, for each $\bbA$, this maps is also surjective (in words: every left adjoint distributor into $\bbC$ is representable by a functor), then we say that $\bbC$ is \emph{Cauchy-complete}. 

Let $\1$ be the $\Q$-category defined by $\1_0=\{*\}$ and $\1(*,*)=1$. A distributor $\phi\:\1\dist\bbC$ is called a (contravariant) \emph{presheaf} on $\bbC$. There is a natural bijection between $\Q$-functors $\1\to\bbC$ and elements of $\bbC_0$. In particular, for any $c\in\bbC_0$ there is a $\Q$-functor $\Delta_c\:\1\to\bbC\:*\mapsto c$ which represents the left adjoint presheaf $(\Delta_c)_*\:\bbC_0\times\1_0\to\Q\:(x,*)\mapsto\bbC(x,c)$. Therefore, by putting
$$c\leq c'\stackrel{\rm def}{\iff}\Delta_c\leq\Delta_{c'}\iff\forall x\in\bbC_0:\bbC(-,c)\leq\bbC(-,c')\iff 1\leq\bbC(c,c')$$
the set $\bbC_0$ becomes an order $(\bbC_0,\leq)$. If both $c\leq c'$ and $c'\leq c$ hold, then we write $c\cong c'$ and we say that these objects of $\bbC$ are \emph{isomorphic}. It is furthermore a result in $\Q$-category theory (which holds in greater generality too) that $\bbC$ is Cauchy-complete if and only if 
$$\Cat(Q)(\1,\bbC)\to\LAdj\Dist(Q)(\1,\bbC)$$
is surjective; in words, $\bbC$ is Cauchy-complete if and only if each left adjoint presheaf on $\bbC$ is representable. 

The importance of Cauchy-complete $\Q$-categories was made very clear in Lawvere's seminal paper \cite{Lawvere} on the subject, via its relation to Cauchy sequences. We shall briefly recall a small portion of this, using Hofmann and Reis \cite[Section 4.3]{hofrei13} as reference.

Given a sequence $x=(x_n)_{n\in\N}$ in a $\Q$-category $\bbC$, we define
$$C_x:=\bigvee\limits_{N\in\N}\bigwedge\limits_{n\geq N}\bigwedge\limits_{m\geq N}\bbC(x_n,x_m).$$
and say that $x=(x_n)_{n\in\N}$ is a \emph{Cauchy sequence} if $C_x\geq1$.
On the other had, we also define 
$$\phi_x\:\bbC_0\to Q\:y\mapsto\bigvee\limits_{N\in\N}\bigwedge\limits_{n\geq N}\bbC(y,x_n)\quad \textrm{and}\quad \psi_x\:\bbC_0\to Q\:y\mapsto\bigvee\limits_{N\in\N}\bigwedge\limits_{n\geq N}\bbC(x_n,y)$$
and then have for these $Q$-valued predicates that:
\begin{proposition}\label{cauchy-sequences}
For any sequence $x=(x_n)_{n\in\N}$ of objects in a $Q$-category $\bbC$, both $\phi_x$ and $\psi_x$ are $\Q$-enriched distributors. Furthermore, the sequence $x=(x_n)_{n\in\N}$ is Cauchy (i.e.\ $C_x\geq 1$) if and only if $\phi_x\dashv\psi_x$.
\end{proposition}
Thus it makes perfect sense to speak of (convergence of) Cauchy sequences in any $Q$-category $\bbC$, via the representability of the associated adjoint pair of distributors, which is exactly what we shall need to do in the proof of Proposition \ref{contr-adj} further on.

\subsection{Examples of \texorpdfstring{$Q$}{Q}-enriched categories}\label{examples-Q-cat}

In the rest of the paper, our examples of $\Q$-enriched categories will be:
\begin{example}[Ordered sets]\label{order}
The simplest (non-trivial) example of a quantale is the two-element Boolean algebra $\Q=(\{0,1\},\vee,\wedge,1)$. In this case, a $\Q$-category is precisely an ordered set $(P,\leq)$, i.e.\ a set $P$ equipped with a binary relation $\leq$ whose characteristic function $P\times P\to\{0,1\}\:(x,y)\mapsto\lb x\leq y\rb$ satisfies the following axioms:
\begin{enumerate}[label=(\arabic*)]
\item $\lb x\leq y\rb\ \wedge\ \lb y\leq z\rb\ \leq\ \lb x\leq z\rb$\ ,
\item $1\leq\ \lb x\leq x\rb$.
\end{enumerate}
(This order-relation need not be anti-symmetric; some call this a ``preorder''.) A $\Q$-functor between such $\Q$-categories is a monotone map between ordered sets. It is well-known (and easy to  verify) that every ordered set is, viewed as an enriched category, Cauchy-complete.
\end{example}
\begin{example}[Metric spaces]\label{metric}
Let $\Q=([0,\infty],\bigwedge,+,0)$ be Lawvere's quantale of extended positive real numbers, i.e.\ it is the segment $[0,\infty]$ (with $\infty$ included) with the converse (!) of the natural (linear) order, and with the sum as binary operation. 
As pointed out by Lawvere \cite{Lawvere}, a $\Q$-category is precisely a \emph{generalised metric space} $(X,d)$, that is, a set $X$ together with a distance function $d\:X\times X\to[0,\infty]$
such that
\begin{enumerate}[label=(\arabic*)]
\item $d(x,y)+d(y,z)\geq d(x,y)$,
\item $0\geq d(x,x)$.
\end{enumerate}
The adjective ``generalized'' here indicates that such a metric need not be finitary (so $d(x,y)=\infty$ is allowed) nor symmetric (so $d(x,y)\neq d(y,x)$ is allowed), nor separated (so $d(x,y)=0=d(y,x)$ for $x\neq y$ is allowed). A $\Q$-functor between such $\Q$-categories is a non-expanding map between (generalized) metric spaces. Lawvere \cite{Lawvere} famously showed that a metric space is Cauchy-complete as enriched category if and only if all Cauchy sequences (in the usual sense for metric spaces) converge.
\end{example}
\begin{example}[Fuzzy orders]\label{fuzzyorder}
A so-called \emph{left-continuous $t$-norm} is precisely a commutative and integral quantale whose underlying (linear) suplattice is $([0,1],\bigvee)$ (see e.g.\ \cite{Klement13,stu13}); the multiplication of such a quantale is then typically written as $x * y$. Examples include $x*y=xy$ (the ``product $t$-norm''), $x*y=\min\{x,y\}$ (the ``minimum $t$-norm'') and $x*y=\max\{x+y-1,0\}$ (the ``Lukasiewicz $t$-norm''); in fact, every \emph{continuous} $t$-norm (meaning that the multiplication is a continuous function) is in a precise sense an amalgamation of these three (see e.g.\ \cite{hadzicpap}). These quantales are the corner stone of ``fuzzy'' logic: the truth values in this logic can vary between $0$ and $1$, conjunction is computed with $*$, and implication is computed with the adjoint to multiplication. A category enriched in a left-continuous $t$-norm $([0,1],\bigvee,*,1)$ thus consists of a set $P$ together with a map $P\times P\to[0,1]\:(x,y)\mapsto \lb x\leq y\rb$ satisfying
\begin{enumerate}[label=(\arabic*)]
\item $\lb x\leq y\rb\ *\ \lb y\leq z\rb \ \leq\ \lb x\leq z\rb$\ ,
\item $1\leq\ \lb x\leq x\rb$.
\end{enumerate}
Following \cite{Zadeh71,ovchinnikov,Coppola,laizhang}, we call this a \emph{fuzzy (pre)order}: the truth value $\lb x\leq y\rb\in[0,1]$ is interpreted as ``the extent to which $x\leq y$ holds in $P$''. A $\Q$-functor between such $\Q$-categories is a map between fuzzy preorders that does not decrease the value of the ``fuzzy'' order. By Theorem 4.19 of \cite{hofrei13} (and the definition of Cauchy sequence in a $\Q$-category recalled above) it follows that a fuzzy order is categorically Cauchy-complete if and only if all Cauchy sequences (in the usual sense for fuzzy orders, see Definition 4.1 in \cite{Coppola}) converge.
\end{example}
\begin{example}[Probabilistic metric space]\label{probametric}
Fix a \emph{left-continuous} $t$-norm $([0,1],\bigvee,*,1)$. It was shown by Hofmann and Reis \cite{hofrei13}, and further explained in \cite{EklGutHohKort}, that the set
$$\Delta=\{f:[0,\infty]\to[0,1]\mid f(t)=\bigvee_{s<t}f(s)\}$$
of so-called distance distributions\footnote{Because domain and codomain are continuous lattices, these are precisely the lower semicontinuous functions, see \cite[Proposition II-2.1]{Gierz}; and because domain and codomain are complete linear orders, these are precisely the supremum-preserving maps, see \cite[Example 2.1.10]{EklGutHohKort}.} is a
quantale for pointwise suprema, with the convolution product
$$(f*g)(t)=\bigvee_{r+s=t}f(r)*g(s)$$
as binary operation, and
$$e(t)=\left\{\begin{array}{l}
0\mbox{ if }t=0,\\
1\mbox{ else }
\end{array}\right.$$
as two-sided unit. Indeed, it is shown in \cite[Examples 2.1.10 and 2.3.36]{GutGarHohKub17} that the quantale $\Q=(\Delta,\bigvee,*,e)$ is the tensor product in the category of suplattices, as well as the coproduct in the category of commutative quantales, of the Lawvere quantale $([0,\infty],\bigwedge,+,0)$ and the left-continuous $t$-norm $([0,1],\bigvee,*,1)$.
A $\Q$-category has been called a \emph{(generalized) probabilistic metric space} by some \cite{hofrei13,HeLiaShen19}, and a \emph{(generalized) fuzzy metric space} by others \cite{KraMic75,GeoVee94}; it consists of a set $X$ together with a probabilistic distance function $d\:X\times X\times[0,\infty]\to[0,1]$
such that
\begin{enumerate}[label=(\arabic*),start=0]
\item $d(x,y,t)=\bigvee_{s<t}d(x,y,s)$,
\item $d(x,x,t)=1$ for $t>0$,
\item $d(x,y,r)*d(y,z,s)\leq d(x,z,r+s)$.
\end{enumerate}
Such an object is often denoted $(X,d,*)$, to stress the importance of the $t$-norm. The intended meaning of $d(x,y,t)$ is that it expresses ``the probability that the distance from $x$ to $y$ is strictly less than $t$''. (Again, we do not insist on finiteness, symmetry or separatedness for such a space, each of which can be expressed suitably; see also \cite{schweizersklar}.) A $\Q$-enriched functor is a map between such spaces that does not decrease such probabilistic distances. Hofmann and Reiss \cite{hofrei13} proved that a probabilistic metric space is categorically Cauchy-complete if and only if all Cauchy sequences (as traditionally defined in probabilistic metric spaces, see \cite{chai09,hofrei13}) converge.
\end{example}

\section{Fixpoints for contractions on \texorpdfstring{$Q$}{Q}-categories}\label{fix}

\subsection{Contractions on a \texorpdfstring{$Q$}{Q}-enriched category}

Let $\Q$ be any quantale (and write $0$ for its bottom element), and $\bbC$ any $\Q$-enriched category.
\begin{definition}\label{def-contr}
If $\varphi:\Q\to\Q$ and $f:\bbC_0\to\bbC_0$ are maps such that 
\begin{enumerate}
\item $\varphi(t)\geq t$ for all $t\in Q$, 
\item if $\varphi(t)=t$ then $t=0$ or $1\leq t$,
\item for all $x,y\in\bbC$, $\bbC(fx,fy)\geq\varphi(\bbC(x,y))$,
\end{enumerate}
then we say that $f$ is a $\varphi$-contraction, and we say that $\varphi$ is a control function for $f$.
\end{definition}
A control function $\varphi$ is thus bigger than the identity function on the whole of $Q$, and strictly so except possibly in $t=0$ or $t\geq 1$. Note too that a $\varphi$-contraction $f$ is always a $Q$-functor $f\:\bbC\to\bbC$, but not every $Q$-functor is $\varphi$-contractive for some control function $\varphi$. 

We wish to investigate the possible fixpoints of such contractions. Let us first make this formal:
\begin{definition}
Let $f\:\bbC\to\bbC$ be a $Q$-functor. A fixpoint for $f$ is an $u\in\bbC$ such that $fu\cong u$ in $\bbC$, that is to say, we have both $1\leq\bbC(fu,u)$ and $1\leq\bbC(u,fu)$.
\end{definition}
In general, such fixpoints are of course not unique. However, if $f$ is a $\varphi$-contraction, and both $fu\cong u$ and $fu'\cong u'$ hold, then it follows from the triangular inequality in $\bbC$ that
\begin{align*}
\bbC(u,u')
&\geq\bbC(u,fu)\circ\bbC(fu,fu')\circ\bbC(fu',u')\\
& \geq 1\circ\bbC(fu,fu')\circ 1 \\
&=\bbC(fu,fu')\\
&\geq\varphi(\bbC(u,u'))\\
&\geq\bbC(u,u')
\end{align*}
Since $\varphi(t)>t$ for all $0\neq t\not\geq1$, we must have $\bbC(u,u')=0$ or $\bbC(u,u')\geq 1$. Exchanging $u$ and $u'$ one sees that also $\bbC(u',u)=0$ or $\bbC(u',u)\geq 1$. Hence there are exactly four possibilities:
$$\left\{\begin{array}{l}
\bbC(u,u')\geq 1\\
\bbC(u',u)\geq 1
\end{array}
\right.
\mbox{ or }
\left\{\begin{array}{l}
\bbC(u,u')\geq 1\\
\bbC(u',u)=0
\end{array}
\right.
\mbox{ or }
\left\{\begin{array}{l}
\bbC(u,u')=0\\
\bbC(u',u)\geq 1
\end{array}
\right.
\mbox{ or }
\left\{\begin{array}{l}
\bbC(u,u')=0\\
\bbC(u',u)=0
\end{array}
\right.$$
Under mild assumptions on $\bbC$ we can now formulate uniqueness results for fixpoints.
\begin{proposition}\label{Uniqueness-Fix}
Let $\bbC$ be a $\Q$-category all of whose homs are non-zero, and $f\:\bbC\to\bbC$ any $\varphi$-contraction. If $fu\cong u$ and $fu'\cong u'$ then $u\cong u'$.
\end{proposition}
\begin{proof}
In the four possible cases above, only the first is compatible with non-zero homs in $\bbC$.
\end{proof}
For $\Q$-categories with homs that can be equal to 0, we have a different result.
\begin{proposition}\label{Generalized-Uniqueness-Fix}
Let $\bbC$ be symmetric $\Q$-category (meaning that $\bbC(x,y)=\bbC(y,x)$ for all $x,y\in\bbC$) and $f\:\bbC\to\bbC$ any $\varphi$-contraction. If $fu\cong u$ and $fu'\cong u'$ then either $u\cong u'$ or $\bbC(u,u')=0$.
\end{proposition}
\begin{proof}
In the four possible cases above, only the first and the last are compatible with symmetry in $\bbC$.
\end{proof}
Reckoning that any symmetric $Q$-category decomposes as a categorical sum of symmetric subcategories, each of which has all homs non-zero, the latter Proposition says that any two distict fixpoints of $f\:\bbC\to\bbC$ must be in different summands of $\bbC$.

\subsection{From contractions to adjoint presheaves}

Given any $\varphi$-contraction $f$ on a $\Q$-category $\bbC$ and an object $x\in\bbC_0$, it follows from Proposition \ref{cauchy-sequences} that the sequence $(f^nx)_{n\in\N}$ determines two distributors, 
$$\phi_{x,f}\:\1\dist\bbC\quad\mbox{ and  }\quad\psi_{x,f}\:\bbC\dist\1,$$ 
with elements
$$\phi_{x,f}(y)=\bigvee_{N\in\N}\bigwedge_{n\geq N}\bbC(y,f^nx)\quad\mbox{ and  }\quad\psi_{x,f}(y)=\bigvee_{N\in\N}\bigwedge_{n\geq N}\bbC(f^nx,y).$$
We now wish to identify sufficient conditions on $\Q$ and $\varphi\:Q\to Q$ in order to prove an adjunction between these distributors. 

To that end, we first recall some pertinent definitions from \cite{Gierz}. Let $L$ be a complete lattice. A subset $D\subseteq L$ is directed if it is non-empty and, for any $x,y\in D$ there exists a $z\in D$ such that $x\vee y\subseteq z$. For two elements $a,b\in L$ we write $a\ll b$, and we say that $a$ is \emph{way below} $b$, if, for every directed subset $D\subseteq L$, $b\leq\bigvee D$ implies the existence of a $d\in D$ such that $a\leq d$. 
\begin{definition}
We say that a complete lattice $L$ is continuous if, for each $a\in L$,
$$a=\bigvee\{u\in L\mid u\ll a\}.$$
\end{definition}
It is well-known that every continuous lattice is meet-continuous---meaning that (binary) meets distribute over (all) directed suprema. Finally, we shall be interested in a weak variant\footnote{A function $f\:L\to M$ between complete lattices is lower-semicontinuous if the sup-inf condition in Definition \ref{def-phi} holds for all \emph{nets} in $L$ (i.e.\ a family of elements indexed by a directed poset).} of \emph{lower-semicontinuity}:
\begin{definition}\label{def-phi}
We say that a function $\varphi:L\to M$ between complete lattices is sequentially lower-semicontinuous if, for any sequence $(t_n)_{n\in\N}$ in $L$,
\begin{equation*}
\varphi(\bigvee_{N\in\N}\bigwedge_{n\geq N}t_n)\leq\bigvee_{N\in\N}\bigwedge_{n\geq N}\varphi(t_n).
\end{equation*}
\end{definition}
Taking inspiration from the ``metric'' case discussed in \cite{boywan69}, we can now prove:
\begin{proposition}\label{contr-adj}
Let $\Q$ be a quantale whose underlying complete lattice is continuous\footnote{It is tempting to speak of a \emph{continuous quantale}, yet this terminology is in conflict with that of \emph{continuous $t$-norm}. Indeed, the underlying lattice of any $t$-norm is the continuous lattice $[0,1]$, yet not every $t$-norm is continuous (as a function in two variables). So we shall stick to the somewhat cumbersome ``quantale with underlying continuous lattice''.}, and let $f:\bbC\to\bbC$ be a $\varphi$-contraction on a $\Q$-category for which the control function $\varphi\:Q\to Q$ is sequentially lower-semicontinuous. For any $x\in\bbC_0$ such that $\bbC(x,fx)\neq0\neq\bbC(fx,x)$ we have that $\phi_{x,f}\dashv\psi_{x,f}$.
\end{proposition}
\begin{proof} 
Putting $C_{x,f}:=\bigvee_{N\in\N}\bigwedge_{n\geq N}\bigwedge_{m\geq N}\bbC(f^nx,f^mx)\in Q$, we recall from Proposition \ref{cauchy-sequences} that $\phi_{x,f}\dashv\psi_{x,f}$ if and only if $C_{x,f}\geq 1$. We shall show that $C_{x,f}\not\geq1$ leads to a contradiction. 

(i) Picking an $x\in\bbC_0$ such that $\bbC(x,fx)\neq0\neq\bbC(fx,x)$, we put $c_n:=\bbC(f^nx,f^{n+1}x)\in\Q$ for all $n\in\N$. By assumption, $0<c_0\leq1$ and the conditions on $\varphi$ imply that $c_0\leq\varphi(c_0)\leq c_1$. Repeating the argument we find that $c_n\leq\varphi(c_n)\leq c_{n+1}$, so the sequence is increasing and strictly above $0$. Therefore we can compute, using the conditions on $\varphi$, that:
\begin{align*}
\bigvee_{N\in\N}c_N 
&=\bigvee_{N\in\N}c_{N+1} \\
&=\bigvee_{N\in\N}\bigwedge_{n\geq N} c_{n+1}\\
&\geq\bigvee_{N\in\N}\bigwedge_{n\geq N}\varphi(c_n)\\
&\geq\varphi(\bigvee_{N\in\N}\bigwedge_{n\geq N}c_n)\\
&=\varphi(\bigvee_{N\in\N}c_N)\\
&\geq\bigvee_{N\in\N}c_N
\end{align*}
We thus find a fixpoint of $\varphi$ which is not $0$, so it must satisfy $1\leq\bigvee_{N\in\N}c_N$.

(ii) Similarly, the sequence $(a_n:=\bbC(f^{n+1}x,f^nx))_{n\in\N}$ must also satisfy $1\leq\bigvee_{\in\N}a_n$. 

(iii) Next, suppose that $1\not\leq C_{f,x}$; by continuity of the underlying complete lattice of $Q$, this means that there exists an $\epsilon\ll 1$ such that $\epsilon\not\leq C_{f,x}$ (and so in particular $\epsilon\neq0$). Using the definition of $C_{f,x}$ as a sup-inf, we may infer:
\begin{align*}
\epsilon\not\leq\bigvee_{k\in\N}\left(\bigwedge_{n\geq k}\bigwedge_{m\geq k}\bbC(f^nx,f^mx) \right)
& \Longrightarrow \forall k\in\N: \epsilon\not\leq\bigwedge_{n\geq k}\bigwedge_{m\geq k}\bbC(f^nx,f^mx)\\
& \Longrightarrow\forall k\in\N,\exists n_k,m_k\geq k: \epsilon\not\leq\bbC(f^{n_k}x,f^{m_k}x)
\end{align*}
In the last line above, it cannot be the case that $m_k=n_k$, because otherwise $\bbC(f^{n_k}x,f^{n_k}x)\geq 1$ (by the ``identity'' axiom for the $Q$-category $\bbC$), which would then also be above $\epsilon\ll 1$. So suppose that $n_k<m_k$, then we can replace $m_k$ by
$$m'_k:=\min\{m>n_k\mid\epsilon\not\leq\bbC(f^{n_k}x,f^mx)\}$$
and so we still have $\epsilon\not\leq\bbC(f^{n_k}x,f^{m'_k}x)$, but now we know also that $\epsilon\leq\bbC(f^{n_k}x,f^{m_k-1}x)$. Similarly, if $n_k>m_k$ then we may replace $n_k$ by
$$n'_k:=\min\{n> m_k\in\N\mid\epsilon\not\leq\bbC(f^nx,f^{m_k}x)\}$$
and we still have $\epsilon\not\leq\bbC(f^{n'_k}x,f^{m_k}x)$, but now we know also that $\epsilon\leq\bbC(f^{n'_k-1}x,f^{m_k}x)$. That is to say, we can always pick $n_k,m_k\geq k$ to ensure that
$$\epsilon\not\leq\bbC(f^{n_k}x,f^{m_k}x)\mbox{ and }
\left\{\begin{array}{rl}
\mbox{ either }&\bbC(f^{n_k}x,f^{m_k-1}x)\geq\epsilon\quad(A)\\
\mbox{ or }&\bbC(f^{n_k-1}x,f^{m_k}x)\geq\epsilon\quad(B)\end{array}\right.$$
Now denote, for each such pick of $n_k,m_k\geq k\in\N$,
$$d_k:=\bbC(f^{n_k}x,f^{m_k}x);$$
and let us insist that $\epsilon\not\leq d_k$ for all $k\in\N$. In case condition (A) holds for $d_k$, then in particular $m_k>n_k$ so $m_k\geq 1$, and we can use the ``composition'' axiom in $\bbC$ to get
\begin{align*}
\epsilon\circ c_{m_k-1}
& \leq \bbC(f^{n_k}x,f^{m_k-1}x)\circ \bbC(f^{m_k-1}x,f^{m_k}x) \\
& \leq \bbC(f^{n_k}x,f^{m_k}x) \\
& = d_k
\end{align*}
In case condition (B) holds for $d_k$ we can similarly prove that
$$ a_{n_k-1}\circ\epsilon\leq d_k.$$
Hence, using in $(*)$ that a continuous lattice is always meet-continuous, and that both sequences
$$
\left(\bigwedge\{d_k\mid k\geq N\mbox{ and }(A)\mbox{ holds}\}\right)_{N\in\N}
\mbox{ and }
\left(\bigwedge\{d_k\mid k\geq N\mbox{ and }(B)\mbox{ holds}\}\right)_{N\in\N}$$
are increasing, we may compute that 
\begin{align*}
\bigvee_{N\in\N}\bigwedge_{k\geq N}d_k
& =\bigvee_{N\in\N_0}\bigwedge_{k\geq N}d_k \\
& =\bigvee_{N\in\N_0}\left(\bigwedge\{d_k\mid k\geq N\mbox{ and }(A)\mbox{ holds}\}\wedge\bigwedge\{d_k\mid k\geq N\mbox{ and }(B)\mbox{ holds}\}\right)\\
  &\stackrel{(*)}{=}\left(\bigvee_{N\in\N_0}\bigwedge\{d_k\mid k\geq N\mbox{ and }(A)\mbox{ holds}\}\right)\\
&\hspace{5cm}\wedge\left(\bigvee_{N\in\N_0}\bigwedge\{d_k\mid k\geq N\mbox{ and }(B)\mbox{ holds}\}\right)\\
  & \geq\left(\bigvee_{N\in\N_0}\bigwedge\{\epsilon\circ c_{m_k-1}\mid k\geq N\mbox{ and } (A)\mbox{ holds}\}\right)\\
&\hspace{5cm}\wedge\left(\bigvee_{N\in\N_0}\bigwedge\{a_{n_k-1}\circ\epsilon\mid k\geq N\mbox{ and }(B)\mbox{ holds}\}\right)\\
  & \geq\left(\bigvee_{N\in\N}\bigwedge_{m\geq N}\epsilon\circ c_m\right)\wedge\left(\bigvee_{N\in\N}\bigwedge_{m\geq N}a_n\circ\epsilon\right)\\
   &\geq\left(\epsilon\circ(\bigvee_{N\in\N}\bigwedge_{m\geq N}c_m)\right)\wedge\left((\bigvee_{N\in\N}\bigwedge_{m\geq N}a_n)\circ\epsilon\right)\\
& =\left(\epsilon\circ(\bigvee_{N\in\N}c_N)\right)\wedge\left((\bigvee_{N\in\N}a_N)\circ\epsilon\right) \\
&=(\epsilon\circ 1)\wedge(1\circ\epsilon) \\
&=\epsilon
\end{align*}
So, even though $\epsilon\not\leq d_k$ (for all $k\in\N$), we do have that $0\neq\epsilon\leq\bigvee_{N\in\N}\bigwedge_{k\geq N}d_k$. 

(iv) Using the ``composition'' axiom in $\bbC$, we have for every $k\geq N\in\N$ (recall that $n_k,m_k\geq k$ too) that
$$d_k\geq c_{n_k}\circ\bbC(f^{n_k+1}x,f^{m_k+1}x)\circ a_{m_k}\geq c_{n_k}\circ\varphi(d_k)\circ a_{m_k}\geq c_N\circ\varphi(d_k)\circ a_N$$
and so we may compute that
\begin{align*}
\bigvee_{N\in\N}\bigwedge_{k\geq N}d_k
&\geq \bigvee_{N\in\N}\bigwedge_{k\geq N}(c_N\circ\varphi(d_k)\circ a_N)\\
&\geq \bigvee_{N\in\N}\left(c_N\circ(\bigwedge_{k\geq N}\varphi(d_k))\circ a_N\right)\\
&\stackrel{(*)}{=}\left(\bigvee_{N\in\N}c_N\right)\circ\left(\bigvee_{N\in\N}\bigwedge_{n\geq N}\varphi(d_k)\right)\circ\left(\bigvee_{N\in\N}a_N\right)\\
&=1\circ\left(\bigvee_{N\in\N}\bigwedge_{k\geq N}\varphi(d_k)\right)\circ 1\\
&\geq \varphi(\bigvee_{N\in\N}\bigwedge_{n\geq N}d_k)\\
&\geq \bigvee_{N\in\N}\bigwedge_{n\geq N}d_k
\end{align*}
where in $(*)$ we use that the involved sequences are increasing\footnote{For two sequences $(a_n)_{n\in\N}$ and $(b_n)_{n\in\N}$ of elements in $Q$, distributivity of product over suprema  in $Q$ assures that $(\bigvee_na_n)\circ(\bigvee_mb_m)=\bigvee_{n,m}(a_n\circ b_m)$. However, when both sequences are increasing, i.e.\ $n\leq n'$ implies $a_n\leq a_{n'}$ and $b_n\leq b_{n'}$, then this is further equal to $\bigvee_n(a_n\circ b_n)$. The argument obviously extends to three increasing sequences.}. This means that $\bigvee_{N\in\N}\bigwedge_{k\geq N}d_k$ is a fixpoint of $\varphi$ which -- as we showed earlier -- is not $0$, so we must have $1\leq\bigvee_{N\in\N}\bigwedge_{k\geq N}d_k$.

(v) Since $\epsilon\ll 1\leq\bigvee_{N\in\N}\bigwedge_{k\geq N}d_k$, and the latter supremum is directed, by continuity of $Q$ there must exist an $N_0\in\N$ such that $\epsilon\leq\bigwedge_{k\geq N_0}d_k$. Yet, we established earlier that $\epsilon\not\leq d_k$ for all $k\in\N$. This is the announced contradiction.
\end{proof}
For simplicity's sake, we asked in the above Proposition that the control function $\varphi\:Q\to Q$ of the contraction $f\:\bbC\to\bbC$ satisfies $\varphi(\bigvee_{N\in\N}\bigwedge_{n\geq N}t_n)\leq \bigvee_{N\in\N}\bigwedge_{n\geq N}\varphi(t_n)$ \emph{for any sequence} $(t_n)_{n\in\N}$ in $\Q$. Yet, in the proof, we really only need such inequalities for sequences whose elements are ``homs'' in $\bbC$ (i.e.\ elements in the image of $\bbC\:\bbC_0\times\bbC_0\to Q$). Similarly, in Definition \ref{def-contr} we asked $\varphi$ to be strictly increasing on all $t\in Q$, except possibly in $t=0$ or $t\geq 1$. But again, in the proof above we really only need this strictness in those particular elements of $Q$ which are suprema, or suprema of infima, of ``homs'' in $\bbC$. Finally, also the condition that the quantale $Q$ has a continuous underlying lattice can be weakened: in the proof we really only use that $1=\bigvee\{u\in Q\mid u\ll 1\}$, so ``continuity at $1\in Q$'' suffices.

\subsection{Fixpoint for a contraction on a Cauchy-complete \texorpdfstring{$Q$}{Q}-category}

In the above Subsection we discovered sufficient conditions for a $\varphi$-contraction $f\:\bbC_0\to\bbC_0$ to determine adjoint distributors. If the $\Q$-category $\bbC$ is Cauchy-complete, this adjoint pair is represented by an object of $\bbC$. We will now show that this representing object is a fixpoint for the contraction.
\begin{proposition}\label{adj-fixpt}
Let $\Q$ be any quantale and $f\:\bbC\to\bbC$ any $Q$-functor on a Cauchy-complete $\Q$-category. If there exists an $x\in\bbC_0$ such that $\phi_{x,f}\dashv\psi_{x,f}$ then $f$ has a fixpoint.
\end{proposition}
\begin{proof}
By Cauchy-completeness of $\bbC$, the presheaves $\phi_{x,f}$ and $\psi_{x,f}$ are representable; so suppose that $\phi_{x,f}=\bbC(-,u)$ and $\psi_{x,f}=\bbC(u,-)$ for some $u\in\bbC_0$.
Now we can compute that
\begin{align*}
\bbC(fu,u)
&=\phi_{x,f}(fu)\\
&=\bigvee_{N\in\N}\bigwedge_{n\geq N}\bbC(fu,f^nx)\\
&=\bigvee_{N\in\N_0}\bigwedge_{n\geq N}\bbC(fu,f^nx)\\
&\stackrel{(*)}{\geq}\bigvee_{N\in\N}\bigwedge_{n\geq N}\bbC(u,f^nx)\\
&=\phi_{x,f}(u)\\
&=\bbC(u,u)\\
&\geq1
\end{align*}
using the ``functoriality'' axiom for $f$ in $(*)$. Similarly one computes that $\bbC(u,fu)\geq1$. Therefore we have both $u\geq fu$ and $fu\geq u$ in (the underlying order of) $\bbC$, which means that $u\cong fu$, as wanted.
\end{proof}
Putting Propositions \ref{contr-adj} and \ref{adj-fixpt} together, we arrive at:
\begin{theorem}[Fixpoint theorem]\label{fixpt}
Let $\Q$ be quantale whose underlying lattice is continuous, and let $f\:\bbC\to\bbC$ a $\varphi$-contraction on a Cauchy-complete $\Q$-category, for which the control function $\varphi\:Q\to Q$ is sequentially lower-semicontinuous. If there exists an $x\in\bbC_0$ such that $\bbC(x,fx)\neq0\neq\bbC(fx,x)$ then $f$ has a fixpoint, namely the object representing the adjunction $\phi_{x,f}\dashv\psi_{x,f}$.
\end{theorem}
In the above Theorem, the obtained fixpoint depends on the element $x\in\bbC$ chosen such that $\bbC(x,fx)\neq0\neq\bbC(fx,x)$. However, let us recall that Propositions \ref{Uniqueness-Fix} and \ref{Generalized-Uniqueness-Fix} provide mild conditions on $\bbC$ to make the fixpoint of a contraction unique. Let us also repeat that the sufficient conditions can be weakened somewhat, as explained in the comment after Proposition \ref{contr-adj}.

\section{Examples and counterexamples}\label{examples}

The examples in this section show how Theorem \ref{fixpt} generalizes known fixpoint theorems from the literature, and provides new ones too. Also, we mention a counterexample to show that the conditions cannot be weakened unless supplementary conditions are considered.

\subsection{Orders}

The two-element boolean algebra being a continuous lattice, the quantale $\Q=(\{0,1\},\vee,\wedge,1)$ satisfies the condition in Theorem \ref{fixpt}, so this Theorem can potentially say something about ordered sets. Note that the functions
$$\varphi_1\:\{0,1\}\to\{0,1\}\:0\mapsto 0,1\mapsto 1\quad\mbox{ and }\quad\varphi_2\:\{0,1\}\to\{0,1\}\:0\mapsto 1,1\mapsto 1$$
are the only possible control functions (according to Definition \ref{def-contr}). A map $f\:(P,\leq)\to(P,\leq)$ is a $\varphi_1$-contraction if and only if $f$ is monotone; and it is a $\varphi_2$-contraction if and only if $f$ is essentially constant ($fx\cong fy$ for all $x,y\in P$). Any non-empty ordered set is Cauchy-complete as a $\Q$-enriched category. It is part of the hypotheses in Theorem \ref{fixpt} that there exists an $x\in P$ such that $x\leq fx$ and $fx\leq x$; in other words, \emph{by hypothesis} there exists a fixpoint $fx\cong x$. Of course this makes the conclusion of the Theorem (namely, the existence of a fixpoint) trivial! Moreover, the fixpoint that is constructed in the proof (as an object representing a left adjoint presheaf) is in this particular case precisely isomorphic to the fixpoint given as hypothesis. So, for the two-element Boolean algebra, Theorem \ref{fixpt} does not give any result; the Theorem can thus only be meaningful for ``richer'' quantales. (We hasten to add that there exist of course very important fixpoint theorems for ordered sets; but these usually require more stringent completeness conditions on the ordered set and/or more stringent continuity conditions on the map. See e.g.\ \cite{Gierz}.)

\subsection{Metric spaces}

Lawvere's quantale $\Q=([0,\infty],\bigwedge,+,0)$ is linear, and therefore continuous\footnote{Any complete linear lattice $L$ is completely distributive and (therefore) also continuous. 
In fact, we have $a\ll b$ if and only if either $a=0$, or $a<b$, or ($a=b$ and $b\neq\bigvee\{x\in L\mid x<b\}$), see \cite{Gierz}.}. It is also an integral quantale: the unit $0$ for the monoid structure is the top element of the lattice (note again that the order on $[0,\infty]$ is the reverse of the natural order!). This makes the notion of contraction in Definition \ref{def-contr} slightly simpler, so by application of Theorem \ref{fixpt} we can produce the following result:
\begin{corollary}\label{cor1}
Let $\varphi\:[0,\infty]\to[0,\infty]$ be an upper-semicontinuous function so that $\varphi(t)<t$ for any $t\not\in\{0,\infty\}$ and $\varphi(0)=0$. Let $f\:X\to X$ be a map on a Cauchy-complete generalized metric space $(X,d)$ such that $d(fx,fy)\leq\varphi(d(x,y))$ for all $x,y\in X$. If there is an $x\in X$ such that $d(x,fx)\neq \infty\neq d(fx,x)$ then the sequence $(f^nx)_{n\in\N}$ converges to a fixpoint of $f$.
\end{corollary}
If $(X,d)$ is a \emph{finitary} generalized metric space, then any $x\in X$ will produce a convergent sequence $(f^nx)_{n\in\N}$; and Proposition \ref{Uniqueness-Fix} implies that all such sequences $(f^nx)_{n\in\N}$ converge to an essentially unique fixpoint of $f$ (unique if the space is also separated).
If $(X,d)$ is a \emph{symmetric} generalized metric space, then any $x\in X$ such that $d(x,fx)\neq\infty$ will produce a convergent sequence $(f^nx)_{n\in\N}$; and Proposition \ref{Generalized-Uniqueness-Fix} implies that any two fixpoints of $f$ are either isomorphic (equal if the space is also separated) or at distance $\infty$ from each other (i.e.\ the space $(X,d)$ decomposes as a categorical sum of two non-empty spaces, and the fixpoints are in different summands). Furthermore, for the reasons explained right after Proposition \ref{contr-adj}, in the above result it is actually enough to require that $\varphi$ is (defined on and) upper-semicontinuous on the closure of $\{d(x,y)\mid x,y\in X\}$. This is how the above result is formulated by Boyd and Wong \cite[Theorem 1]{boywan69} in the particular case that $(X,d)$ is a finitary, symmetric, separated metric space (see also \cite{LechB}):
\begin{example}\label{BWExample}
For $(X,d)$ a Cauchy-complete metric space, let $\varphi\:\overline{\{(d(x,y)\mid x,y\in X\}}\!\to\![0,\infty]$ be an upper-semicontinuous function that maps $0$ to $0$ and so that $\varphi(t)<t$ for any $t\not\in\{0,\infty\}$. Then any map $f\:X\to X$ satisfying $d(fx,fy)\leq\varphi(d(x,y))$ for all $x,y\in X$ has a unique fixpoint, and for any $x\in X$ the sequence $(f^nx)_{n\in\N}$ converges to that fixpoint.
\end{example}
On the other hand, the control function defined by $\varphi(t)=k\cdot t$ for $0<k<1$ certainly satisfies the conditions in Corollary \ref{cor1}, so we find the following particular case:
\begin{example}
Let $f\:X\to X$ be a map on a Cauchy-complete generalized metric space $(X,d)$ for which there exists a $0<k<1$ such that $d(fx,fy)\leq k\cdot d(x,y)$ for all $x,y\in X$. If there is an $x\in X$ such that $d(x,fx)\neq \infty\neq d(fx,x)$ then the sequence $(f^nx)_{n\in\N}$ converges to a fixpoint of $f$.
\end{example}
Obviously, if $(X,d)$ is finitary, symmetric and separated (and hence an ordinary metric space), we find here the well-known Banach Fixpoint Theorem. Finally, we mention that Ackerman \cite{Ackerman16} has produced an example of a \emph{non-expansive} contraction -- whose control function merely satisfies $\varphi(t)\leq t$ instead of $\varphi(t)<t$ for $t\not\in\{0,\infty\}$ -- on a Cauchy-complete metric space which does not have a fixpoint. This shows that this condition on the control map cannot be weakened without strenghtening some other conditions in Corollary \ref{cor1}.

\subsection{Fuzzy orders}

The quantale $\Q=([0,1],\bigvee,*,1)$, where $*$ is a left-continuous t-norm, is linear (thus continuous, see a previous footnote) and integral. Hence, by application of Theorem \ref{fixpt} we find:
\begin{corollary}
Let $\varphi\:[0,1]\to[0,1]$ be a lower-semicontinuous function satisfying $\varphi(t)>t$ for all $0<t<1$, and $\varphi(1)=1$. Suppose that $(P,\lb \cdot\leq\cdot\rb)$ is a complete fuzzy preorder, and that $f\:P\to P$ is a function such that $\lb fx\leq fy\rb\geq\varphi(\lb x\leq y\rb)$. If there exist $x\in P$ such that $\lb x\leq fx\rb\neq0\neq\lb fx\leq x\rb$, then the sequence $(f^n(x))_{n\in\N}$ converges to a fixpoint of $f$.
\end{corollary}
This can be regarded as a (straightforward) generalization of Corollary \ref{cor1} above, since Lawvere's quantale $([0,\infty],\bigwedge,+,0)$ is isomorphic to the product $t$-norm $([0,1],\bigwedge,\cdot,1)$ by the order-reversing map $[0,\infty]\to[0,1]\:t\mapsto\exp(-t)$.

\subsection{Probabilistic metric spaces}

The integral quantale $(\Delta,\bigvee,*,e)$ of distance distributions (wrt.\ a left-continuous t-norm $*$) is completely distributive\footnote{Indeed, the complete distributivity of the underlying suplattices follows from \cite[Theorem 2.1.17]{GutGarHohKub17}, who show that the tensor product of completely distributive complete lattices is completely distributive.}, hence continuous, so we can apply Theorem \ref{fixpt}.
\begin{corollary}
Let $\varphi\:\Delta\to\Delta$ be a lower-semicontinuous function satisfying $\varphi(u)>u$ for all $0<u<e$, and $\varphi(e)=e$. Suppose that $f:X\to X$ is a function on a Cauchy-complete generalized probabilistic metric space $(X,d,*)$ such that $d(fx,fy,t)\geq\varphi(d(x,y,t)$ for all $t$. If there exists an $x\in X$ such that $d(x,fx,t)\neq0\neq d(fx,x,t)$ then $f$ has a fixpoint.
\end{corollary}
It follows furthermore from Proposition \ref{Uniqueness-Fix} that, if $d(x,y,\infty)=1$ for all $x,y\in X$ (i.e.\ the space is \emph{finitary}), then the fixed point is unique.

There are indeed examples of control functions $\varphi\:\Delta\to\Delta$ that the above statement asks for, e.g.\
$$\varphi(u)(t):=\left\{
\begin{array}{ll}
\frac{1}{2}(u(t)+1)&\mbox{ if }0<t\leq\infty\\
0&\mbox{ if }t=0
\end{array}\right.$$
Unfortunately though, the ``Banach control function'' which is appropriate in the setting of probabilistic metric spaces\footnote{In \cite{hadzicpap} the contractions with this control function are called \emph{probabilistic q-contractions}.},
$$\varphi(u)(t)=u(Kt)\mbox{ for some }1<K<\infty,$$
does not satisfy $\varphi(u)\neq u$ for all $0\neq u\neq e$ (e.g.\ the ``almost constant'' functions, defined by $u(t)=u_0$ for $0<u_0<1$ and $0<t\leq\infty$, are fixpoints of $\varphi$). One possible solution (hinted at by a result in \cite{Gra}) would be to work with \emph{finitary} probabilistic metric spaces. These can be seen as categories enriched in the subquantale
$$\Delta^+=\{u\in\Delta\mid u(\infty)=1\}\cup\{0\}$$
of $\Delta$. Restricted to $\Delta^+$, the Banach control function does not have fixpoints other than $0$ and $e$: if $u\in\Delta^+\setminus\{0\}$ satisfies $u(t)=u(Kt)$, then for any $0<t_0<\infty$,
$$1=u(\infty)=u(\bigvee_{n\in\N}K^nt_0)=\bigvee_{n\in\N}u(K^nt_0)=u(t_0),$$
so indeed $u=e$. However, we do not know whether $\Delta^+$ is continuous (we conjecture that it is not), so we do not know whether we can apply Theorem \ref{fixpt} without modifications: this will be a topic of futher investigation.

\section{Conclusion and further work}\label{concl}

With our study of fixpoint theorems for quantale-enriched categories, we exemplified that such results depend not only on the strength of the contraction and the completeness of the space, but also on the algebraic properties of the underlying quantale: any fixpoint theorem results from an equilibrum between those three aspects.
 
In future work, we want to investigate how several examples of fixpoint theorems in the literature (see e.g.\ \cite{Coppola,Gra,hadzicpap}) fit  -- or, perhaps, do not fit -- with our quantale-enriched approach. This could lead to variants on our Theorem \ref{fixpt}, where different algebraic properties of $Q$ are combined with different conditions on the control functions of contractions, or with different completeness conditions on the $Q$-categories (see e.g.\ \cite{wagner97}). We also intend to study fixpoint theorems for \emph{quantaloid}-enriched categories. This generalization, far from trivial, has the benefit to include in particular the theory of partial metric spaces \cite{stu13,stu18} and of probabilistic partial metric spaces \cite{HeLiaShen19}, two areas for which only few fixpoint theorems are known \cite{oltval04,panal,romag12}.

\section*{Acknowledgement}
The authors warmly thank U. Höhle for his comments on an earlier version of the manuscript. Arij Benkhadra was supported by the CNRST Excellence Scholarship no.\ 15UM5R2018 (Morocco) and the Université du Littoral-Côte d'Opale (France).

\end{document}